\documentclass[11pt]{amsart}
\usepackage{amsthm}
\usepackage{cite}
\usepackage{array}
\usepackage{amsmath}
\usepackage{enumerate}
\usepackage{tikz}
\usetikzlibrary{calc}
\usepackage{amssymb}
\usepackage{latexsym}
\usepackage{amsfonts}
\usepackage{color}
\usepackage{mathrsfs}
\usepackage{epsfig,helvet}

\theoremstyle{plain} \numberwithin{equation}{section}
\newtheorem{thm}{Theorem}[section]
\newtheorem{theorem}[thm]{Theorem}
\newtheorem{lemma}[thm]{Lemma}
\newtheorem{corollary}[thm]{Corollary}

\begin{document}
\setcounter{page}{1}

\author[Padhan]{Rudra Narayan Padhan}
\address{Centre for Data Science, Institute of Technical Education and Research  \\
	Siksha `O' Anusandhan (A Deemed to be University)\\
	Bhubaneswar-751030 \\
	Odisha, India}
\email{rudra.padhan6@gmail.com, rudranarayanpadhan@soa.ac.in}

\title[ Hasan and Padhan]{On the dimension of the Schur multiplier of nilpotent
	Lie superalgebras}

\author{IBRAHEM YAKZAN HASAN}
\address{Centre for Data Science, Institute of Technical Education and Research \\
	Siksha `O' Anusandhan (A Deemed to be University)\\
	Bhubaneswar-751030 \\
	Odisha, India}
\email{ibrahemhasan898@gmail.com}

\maketitle

\begin{abstract}
We provide a bound on the dimension of Schur multiplier of a finite dimensional nilpotent Lie superalgebra which is more precise than the previous bounds on the dimension of Schur multiplier of  Lie superalgebra.
\end{abstract}

\section{Introduction}

The theory of Lie Superalgebras has applications in many areas of Mathematics and Theoretical Physics as it can be used to describe supersymmetry. Kac \cite{Kac1977} gives a comprehensive description of the mathematical theory of Lie superalgebras, and establishes the classification of all finite dimensional simple Lie superalgebras over an algebraically closed field of characteristic zero. In the last decade, the theory of Lie superalgebras has evolved remarkably, obtaining many results in representation theory and classification. Most of the results are extensions of well known facts of Lie algebras. In 1904, I. Schur introduced the Schur multiplier and cover of a group in his work on projective representation. Batten \cite{Batten1993} introduced and  investigated  Schur multiplier and cover of a Lie algebra and several authors later followed, see\cite{BS1996,Batten1996,Ellis1991,Ellis1995,Ellis1996,Hardy1998,Hardy2005,org,Kar1987,Moneyhun1994,N2,Niroomand2011,Russo2011,PMF2013}. For a finite dimensional Lie algebra $L$ over a field $\mathbb{F}$  the free presentation of $L$ is $0 \longrightarrow R\longrightarrow F \longrightarrow L \longrightarrow 0$, where $F$ is a free Lie algebra and $R$ is an ideal of $F$. Then the Schur multiplier $\mathcal{M}(L)$ is isomorphic to $F^2\cap R/[F,R]$. Niroomand and Russo \cite{Niroomand2011} proved that for a $c$-step nilpotent Lie algebra $L$ with $\dim L =n$ and $\dim L^2= m\geq1$,  then
$\dim(\mathcal{M}(L))\leq \frac{1}{2}[(n+m-2)(n-m-1)+1.$ Recently, Rai \cite{Rai}improved this bounds, i.e., 
\begin{equation}\label{eq01}
\dim(\mathcal{M}(L))\leq \frac{1}{2}[(n+m)(n-m-1)-\sum_{i=2}^{\min\{c,n-m\}}(n-m-i).
\end{equation}
We note here that this bound also follows from [6, Theorem 1], where Gholami and Saeedi obtained a bound for the dimension of the Schur multiplier of a nilpotent $n$-Lie algebra. The study of Schur multiplier for Lie superalgebra has  recently been started \cite{GKL2015,hasanpadprad,p54,p55,p50,Padhandetec,hesam,ph57}.
Recently, the bound of Niroomand and Russo are analyzed by Nayak as,  $\dim(\mathcal{M}(L))\leq \frac{1}{2}[(n+m+r+s-2)(n+m-r-s-1)]+n+1,$
for a nilpotent Lie superalgebra of dimension $(m\mid n)$ and $\dim L^2=(r\mid s)$ with $r+s\geq1$. In this paper, we find an analogous bound to (\ref{eq01}), improving the bound of Nayak.

\section{Preliminaries}

Let $\mathbb{Z}_{2}=\{\bar{0}, \bar{1}\}$ be a field. A $\mathbb{Z}_{2}$-graded vector space $V$ is simply a direct sum of vector spaces $V_{\bar{0}}$ and $V_{\bar{1}}$, i.e., $V = V_{\bar{0}} \oplus V_{\bar{1}}$. It is also known as  superspace. We consider all vector superspaces and superalgebras are over the field $\mathbb{F}$ (characteristic of $\mathbb{F} \neq 2,3$). Elements in $V_{\bar{0}}$ (resp. $V_{\bar{1}}$) are called even (resp. odd) elements. Non-zero elements of $V_{\bar{0}} \cup V_{\bar{1}}$ are called homogeneous elements. For a homogeneous element $v \in V_{\sigma}$, with $\sigma \in \mathbb{Z}_{2}$ we set $|v| = \sigma$ as the degree of $v$. A  subsuperspace (or, subspace)  $U$ of $V$ is a $\mathbb{Z}_2$-graded vector subspace where  $U= (V_{\bar{0}} \cap U) \oplus (V_{\bar{1}} \cap U)$. We adopt the convention that whenever the degree function appears in a formula, the corresponding elements are supposed to be homogeneous. 

\smallskip 

A Lie superalgebra (see \cite{Kac1977, Musson2012}) is a superspace $L = L_{\bar{0}} \oplus L_{\bar{1}}$ with a bilinear mapping $ [., .] : L \times L \rightarrow L$ satisfying the following identities:
\begin{enumerate}
\item $[L_{\alpha}, L_{\beta}] \subset L_{\alpha+\beta}$, for $\alpha, \beta \in \mathbb{Z}_{2}$ ($\mathbb{Z}_{2}$-grading),
\item $[x, y] = -(-1)^{|x||y|} [y, x]$ (graded skew-symmetry),
\item $(-1)^{|x||z|} [x,[y, z]] + (-1)^{ |y| |x|} [y, [z, x]] + (-1)^{|z| |y|}[z,[ x, y]] = 0$ (graded Jacobi identity),
\end{enumerate}
for all $x, y, z \in L$. It is clear that the  graded Jacobi identity  is equivalent to the equation 
\begin{equation}\label{eq22}
[x,[y, z]]=  [[x,y], z] + (-1)^{|z| |x+y|}[[z,x],y]. 
\end{equation} 
Clearly $L_{\bar{0}}$ is a Lie algebra, and $L_{\bar{1}}$ is a $L_{\bar{0}}$-module. If $L_{\bar{1}} = 0$, then $L$ is just a Lie algebra, but in general, a Lie superalgebra is not a Lie algebra.  A Lie superalgebra $L$ is called abelian if  $[x, y] = 0$ for all $x, y \in L$. Lie superalgebras without even part, i.e., $L_{\bar{0}} = 0$, are  abelian. A subsuperalgebra (or subalgebra) of $L$ is a $\mathbb{Z}_{2}$-graded vector subspace that is closed under bracket operation. The graded subalgebra $[L, L]$, of $L$, is known as the derived subalgebra of $L$. A $\mathbb{Z}_{2}$-graded subspace $I$ is a graded ideal of $L$ if $[I, L]\subseteq I$. The ideal 
\[Z(L) = \{z\in L : [z, x] = 0\;\mbox{for all}\;x\in L\}\] 
is a graded ideal and it is called the {\it center} of $L$. A homomorphism between superspaces $f: V \rightarrow W $ of degree $|f|\in \mathbb{Z}_{2}$, is a linear map satisfying $f(V_{\alpha})\subseteq W_{\alpha+|f|}$ for $\alpha \in \mathbb{Z}_{2}$. In particular, if $|f| = \bar{0}$, then the homomorphism $f$ is called the homogeneous linear map of even degree. A Lie superalgebra homomorphism $f: L \rightarrow M$ is a  homogeneous linear map of even degree such that $f([x,y]) = [f(x), f(y)]$ holds for all $x, y \in L$.  If $I$ is an ideal of $L$, the quotient Lie superalgebra $L/I$ inherits a canonical Lie superalgebra structure such that the natural projection map becomes a homomorphism. The notions of {\it epimorphisms, isomorphisms,} and {\it automorphisms} have obvious meaning. The descending central sequence of a Lie superalgebra $L = L_{\bar{0}} \oplus L_{\bar{1}}$ is defined by $L^{1} = L, L^{k+1} = [L^{k}, L]$ for all $k\geq 1$. If $L^{k+1} = \{0\}$ for some $k+1$, with $L^{k} \neq 0$, then $L$ is called nilpotent with nilpotency class $k$. By $\gamma_i(L)$  we denote to the $L^{k}$.
\smallskip

Throughout this article, for the superdimension of Lie superalgebra $L$, we simply write $\dim L=(m\mid n)$ , where $\dim L_{\bar{0}} = m$ and $\dim L_{\bar{1}} = n$. An abelian Lie superalgebra with $\dim A=(m\mid n)$ is denoted by $A(m \mid n)$.  Now we state some useful known results for further use.

\smallskip

\begin{lemma}\label{th3.3}\cite[See Theorem 3.4]{Nayak2018}
\[\dim \mathcal{M}(A(m \mid n)) = \big(\frac{1}{2}(m^2+n^2+n-m)\mid mn \big).\]
\end{lemma}

\begin{lemma}
Let $L$ be a finite dimensional $c$-step nilpotent Lie superalgebra and let $0 \longrightarrow R\longrightarrow F \longrightarrow L \longrightarrow 0$ be a free presentation of the Lie superalgebra $L$. Then
\begin{equation}\label{eq2.1}
\dim(\gamma_c(L)) + \dim(\mathcal{M}(L)) = \dim(\mathcal{M}(L/\gamma_c(L))) + \dim
\frac{[\gamma_c(F) + R,F]}{[R, F]}
\end{equation}
	
	.
\end{lemma}
\begin{proof}
See the proof of [\cite{p54}, Theorem 3.6]).
\end{proof}

Now, similar to the Lie algebra case, we can here also define the epimorphism $$\lambda_c : \gamma_c(L) \otimes\frac{L}{\gamma_2(L)}\longrightarrow \frac{[\gamma_c(F) + R,F]}{[R, F]},$$
and the epimorphisms
$$\lambda_i: \frac{\gamma_i(L) }{\gamma_{i+1}(L)} \times\frac{L}{\gamma_2(L)}\longrightarrow \frac{[\gamma_i(F) + R,F]}{[\gamma_{i+1}(F) +R, F]}$$
for $2 \leq i \leq c-1.$  Thus

$$ \dim
\frac{[\gamma_c(F) + R,F]}{[R, F]}=\dim (\gamma_c(L) \otimes\frac{L}{\gamma_2(L)})- \dim (ker (\lambda_c)),$$
and 
$$ \dim
\frac{[\gamma_i(F) + R,F]}{[\gamma_{i+1}(F) +R, F]}=\dim (\frac{\gamma_i(L) }{\gamma_{i+1}(L)} \otimes\frac{L}{\gamma_2(L)})- \dim( ker (\lambda_i)),$$
for $2 \leq i \leq c-1.$
Then by Equation \ref{eq2.1}, we have
$$ \dim(\mathcal{M}(L)) = \dim(\mathcal{M}(L/\gamma_c(L)))- \dim(\gamma_c(L))  + \dim( \gamma_c(L) \otimes\frac{L}{\gamma_2(L)})- \dim (ker (\lambda_c))$$
i.e.,
\begin{equation}
 \dim(\mathcal{M}(L)) = \dim(\mathcal{M}(L/\gamma_c(L)))  + \dim(\gamma_c(L))(\dim (\frac{L}{\gamma_2(L)})-1) - \dim (ker (\lambda_c)).
\end{equation}

For $2 \leq i \leq c-1,$ observe that $0 \longrightarrow \gamma_i(L) + R \longrightarrow F \longrightarrow L/\gamma_i(L) \longrightarrow 0 $ is a free presentation of
$L/\gamma_i(L)$ and $L/\gamma_i(L)$ is an $(i-1)$-step nilpotent Lie superalgebra. Thus by repeating equation \ref{eq2.1}
and adding we get
\begin{equation}\label{eq2.4}
	 \dim(\mathcal{M}(L)) = \dim(\mathcal{M}(L/\gamma_2(L)))  + \dim(\gamma_2(L))(\dim (\frac{L}{\gamma_2(L)})-1) - \sum_{i=2}^{c}\dim (ker (\lambda_i)).
\end{equation}

\section{ The Main Results}

\begin{lemma}\label{lem3.1}
	Let $L$ be a Lie superalgebra. Then for any homogeneous elements $x_1, x_2, \ldots, x_{i+1} \in L,$
	
	$$	(-1)^{(|x_1|+\ldots+|x_{i-1}|) |x_{i+1}|}  [[x_1, \ldots, x_{i}]_l,x_{i+1}]+(-1)^{|x_{i+1}||x_i|}[[x_{i+1},[x_1, \ldots, x_{i-1}]_l],x_i]$$

	$$+(-1)^{(|x_{i+1}|+|x_i|)|x_{i-1}|}[[[x_i,x_{i+1}]_r,[x_1, \ldots, x_{i-2}]_l],x_{i-1}]$$

	$$+(-1)^{(|x_1|+ \ldots+ |x_{i-2}|)|x_{i-1}|}[[[x_{i-1},x_i,x_{i+1}]_r,[x_1, \ldots, x_{i-3}]_l],x_{i-2}]$$

	$$+(-1)^{(|x_1|+ \ldots+ |x_{i-3}|)(|x_{i-2}|+|x_{i-1}|)+(|x_{i+1}|+|x_i|)|x_{i-2}|}[[[x_{i-2},x_{i-1},x_i,x_{i+1}]_r,[x_1, \ldots, x_{i-4}]_l],x_{i-3}]$$

    $$+(-1)^{(|x_1|+ \ldots+ |x_{i-4}|)(|x_{i-3}|+|x_{i-2}|+|x_{i-1}|)+(|x_{i+1}|+|x_i|)(|x_{i-3}|+|x_{i-2}|)}[[[x_{i-3},\ldots,x_{i+1}]_r,[x_1,, \ldots, x_{i-5}]_l],x_{i-4}]$$
    
    $$+\ldots+(-1)^{|x_1|(|x_2|+ \ldots+ |x_{i-1}|)+(|x_{i+1}|+|x_i|)(|x_2|+\ldots|x_{i-2}|)}[[x_2,, \ldots, x_{i+1}]_r,x_1]$$
    
     $$+\{(-1)^{(|x_1|+ \ldots+ |x_{i-2}|)|x_i|}-(-1)^{|x_{i-1}||x_{i+1}|}\}(-1)^{(|x_1|+ \ldots+ |x_{i-2}|)|x_{i+1}|}[[x_1, \ldots, x_{i-1}]_l,[x_i,x_{i+1}]]=0$$
for $i \geq 3$, where
   $$[x_1,x_2,\ldots x_i]_r=[x_1,[\ldots[x_{i-2},[x_{i-1},x_i]]\ldots]$$
 and 
     $$[x_1,x_2,\ldots x_i]_l=[\ldots[[x_1,x_2],x_3],\ldots,x_i].$$
\end{lemma}

\begin{proof}
First, we demonstrate  the lemma for $i=3$. Consider the graded Jacoby identity
$$(-1)^{|a||c|}[a,b,c]_l+(-1)^{|b||a|}[b,c,a]_l+(-1)^{|c||b|}[c,a,b]_l=0$$
with $a=[x_1,x_2], ~ b=x_3, ~c=x_4$. Then we get
$$(-1)^{(|x_1|+|x_2|)|x_4|}[[x_1,x_2],x_3,x_4]_l+(-1)^{|x_3|(|x_1|+|x_2|)}[x_3,x_4,[x_1,x_2]]_l+(-1)^{|x_3||x_4|}[x_4,[x_1,x_2],x_3]_l=0.$$
Using the graded Jacoby identity again with $a=x_1,~ b=x_2, ~c=[x_3,x_4]$, gives 
$$(-1)^{|x_1|(|x_3|+|x_4|)}[x_1,x_2,[x_3,x_4]]_l+(-1)^{|x_2||x_1|}[x_2,[x_3,x_4],x_1]_l+(-1)^{(|x_3|+|x_4|)|x_2|}[[x_3,x_4],x_1,x_2]_l=0.$$
By adding the previous two equations, we get 
$$(-1)^{(|x_1|+|x_2|)|x_4|}[[x_1,x_2],x_3,x_4]_l +(-1)^{|x_3||x_4|}[x_4,[x_1,x_2],x_3]_l+(-1)^{(|x_3|+|x_4|)|x_2|}[[x_3,x_4],x_1,x_2]_l $$
	$$+(-1)^{|x_2||x_1|}[x_2,[x_3,x_4],x_1]_l+(-1)^{|x_3|(|x_1|+|x_2|)}[x_3,x_4,[x_1,x_2]]_l+(-1)^{|x_1|(|x_3|+|x_4|)}[x_1,x_2,[x_3,x_4]]_l=0.$$
Now, by graded skew-symmetry, we have 
$$(-1)^{|x_3|(|x_1|+|x_2|)}[x_3,x_4,[x_1,x_2]]_l+(-1)^{|x_1|(|x_3|+|x_4|)}[x_1,x_2,[x_3,x_4]]_l=$$
$$\{(-1)^{|x_1||x_3|}-(-1)^{|x_2||x_4|}\}(-1)^{|x_1||x_4|}[x_1,x_2,[x_3,x_4]]_l=0.$$
Thus the lemma holds for $i=3$. Assume that the lemma is true for some   $i = j \geq 3$, i.e., if we take $a_1,~ a_2,~ \ldots ~a_{j+1}\in L$ then the result holds. Now, by putting $a_1=[x_1,x_2],~ a_2=x_3,~a_3=x_4,~ \ldots ~a_{j+1}=x_{j+2}$, we get
$$	(-1)^{(|x_1|+\ldots+|x_{j}|) |x_{j+2}|}  [[x_1, \ldots, x_{j+1}]_l,x_{j+2}]+(-1)^{|x_{j+1}||x_{j+2}
|}[[x_{j+2},[x_1, \ldots, x_{j}]_l],x_{j+1}]$$

$$+(-1)^{(|x_{i+2}|+|x_{j+1}|)|x_{j}|}[[[x_{j+1},x_{j+2}]_r,[x_1, \ldots, x_{j-1}]_l],x_{j}]$$

$$+(-1)^{(|x_1|+ \ldots+ |x_{j-1}|)|x_{j}|}[[[x_{j},x_{j+1},x_{j+2}]_r,[x_1, \ldots, x_{j-2}]_l],x_{j-1}]$$

$$+(-1)^{(|x_1|+ \ldots+ |x_{j-2}|)(|x_{j-1}|+|x_{j}|)+(|x_{j+2}|+|x_{j+1}|)|x_{j-1}|}[[[x_{j-1},x_{j},x_{j+1},x_{j+2}]_r,[x_1, \ldots, x_{j-3}]_l],x_{j-2}]$$

$$+(-1)^{(|x_1|+ \ldots+ |x_{j-3}|)(|x_{j-2}|+|x_{j-1}|+|x_{j}|)+(|x_{j+2}|+|x_{j+1}|)(|x_{j-2}|+|x_{j-1}|)}[[[x_{j-2},\ldots,x_{j+2}]_r,[x_1,, \ldots, x_{j-4}]_l],x_{j-3}]$$

$$+\ldots+(-1)^{(|x_1|+|x_2|+|x_3|)(|x_4|+|x_5|+\ldots+x_j)+(|x_{j+2}|+|x_{j+1}|)(|x_4|+\ldots+|x_{j-1}|)}[[x_4, \ldots, x_{j+2}]_r,[x_1,x_2]_l,x_3]$$

$$+(-1)^{(|x_1|+|x_2|)(|x_3|+ \ldots+ |x_{j}|)+(|x_{j+2}|+|x_{j+1}|)(|x_3|+\ldots|x_{j-1}|)}[[x_3, \ldots, x_{j+2}]_r,[x_1,x_2]]$$

$$+\{(-1)^{(|x_1|+ \ldots+ |x_{j-1}|)|x_{j+1}|}-(-1)^{|x_{j}||x_{j+2}|}\}(-1)^{(|x_1|+ \ldots+ |x_{j-1}|)|x_{j+2}|}[[x_1, \ldots, x_{j}]_l,[x_{j+1},x_{j+2}]]=0.$$
Applying the Equation \ref{eq22}, on the penultimate term in the previous equation  we have
$$[[x_3, \ldots, x_{j+2}]_r,[x_1,x_2]]=[ [[x_3,\ldots, x_{j+2}]_r,x_1],x_2]]+(-1)^{|x_2|(|x_1|+|x_3|+|x_4|+ \ldots+ |x_{j+2}|)}[[x_2,[x_3,\ldots,x_{j+2}]]_r,x_1].$$
We can see that the lemma is true for $j+1$ by applying this to the preceding equation. Thus the proof follows from induction. 
\end{proof}

\begin{corollary}\label{cor3.2}
Let $\{x_1, x_2, \ldots, x_i, x_{i+1}\}$ be the minimal generating set of the Lie
superalgebra $L$ with $i \geq 2$. Then
$$\phi_{i}(x_1,x_2,\ldots,x_{i+1}):=		(-1)^{(|x_1|+...+|x_{i-1}|) |x_{i+1}|}  \overline{[x_1, \ldots, x_{i}]_l}\otimes \overline{x_{i+1}}$$ 
$$+(-1)^{|x_{i+1}||x_{i}|}\overline{[x_{i+1},[x_1, \ldots, x_{i-1}]_l]}\otimes \overline{x_{i}}$$
$$+(-1)^{(|x_{i+1}|+|x_{i}|)|x_{i-1}|}\overline{[[x_{i},x_{i+1}]_r,[x_1, \ldots, x_{i-2}]_l]}\otimes \overline{x_{i-1}}$$
$$+(-1)^{(|x_1|+ \ldots+ |x_{i-2}|)|x_{i-1}|}\overline{[[x_{i-1},x_{i},x_{i+1}]_r,[x_1, \ldots, x_{i-3}]_l]}\otimes \overline{x_{i-2}}$$
\begin{align*}
	+(-1)^{(|x_1|+ \ldots+ |x_{i-3}|)(|x_{i-2}|+|x_{i-1}|)+(|x_{i+1}|+|x_{i}|)|x_{i-2}|}&\overline{[[x_{i-2},x_{i-1},x_{i},x_{i+1}]_r,}\\
	&\overline{[x_1, \ldots, x_{i-4}]_l]}\otimes \overline{x_{i-3}}
	\end{align*}
\begin{align*}
+(-1)^{(|x_1|+ \ldots+ |x_{i-4}|)(|x_{i-3}|+|x_{i-2}|+|x_{i-1}|)+(|x_{i+1}|+|x_{i}|)(|x_{i-3}|+|x_{i-2}|)}&\overline{[[x_{i-3},\ldots ,x_{i+1}]_r,}\\
	&\overline{[x_1, \ldots, x_{i-5}]_l]}\otimes \overline{x_{i-4}}
\end{align*}
$$+\ldots+(-1)^{|x_1|(|x_2|+ \ldots+ |x_{i-1}|)+(|x_{i+1}|+|x_{i}|)(|x_2|+\ldots|x_{i-2}|)}\overline{[x_2,, \ldots, x_{i+1}]_r}\otimes \overline{x_1} \in ker (\lambda_{i}).$$
\end{corollary}

\begin{proof}
By definition of $\lambda_2$, the result is true for $i=2$. Now, for $i\geq3$, we have from Lemma \ref{lem3.1} 
\begin{align*}
	\phi_{i}(x_1,x_2,\ldots,x_{i+1})&=\\ +\{(-1)^{(|x_1|+ \ldots+ |x_{i-2}|)|x_{i}|}-(-1)^{|x_{i-1}||x_{i+1}|}\}&(-1)^{(|x_1|+ \ldots+ |x_{i-2}|)|x_{i+1}|}\overline{[[x_1,\ldots}\overline{\ldots, x_{i-1}]_l,[x_{i},x_{i+1}]]}\\
	& \in \ker (\lambda_{i}).
\end{align*}
\end{proof}

\begin{theorem}
Let $L$ be a $(m\mid n)$ dimensional $c$-step nilpotent Lie superalgebra with $\dim(\gamma_2(L)) =(r\mid s)$ and $r+s\geq1$. Then
$$\dim(\mathcal{M}(L))\leq \frac{1}{2}[(m + n - r - s)(m + n + r + s) + (n - m - r - 3s)]-\sum_{i=2}^{l}(m+n-r-s-i),$$
where $l=\min\{c,m+n-r-s\}.$

\end{theorem}
\begin{proof}
Let $U=\{x_1,x_2,\ldots,x_{m-r};x_{m-r+1},\ldots,x_{m+n-r-s}\}$  be the minimal generating set of $L$, and let $i$ be an arbitrary fixed number such that $i\leq \min\{m+n-r-s,c\}=l$. Since $i\leq c $, the space $\gamma_{i}(L)/\gamma_{i+1}(L)$ is non-trivial. As a result, there  is a non-zero element of $\gamma_{i}(L)$, say $z^*=[z_1,z_2,\ldots,z_{i}]$ such that $z^*\notin \gamma_{i+1}(L)$. Since $i\leq m+n-r-s$, there exist $n+m-r-s-i$ elements in the set $U/\{z_1,z_2,\ldots,z_{i}\}$, say $y_k$  for $k=1,\ldots,n+m-r-s-i.$ As $z^*\notin \gamma_{i+1}(L)$ and $y_k \notin \{z_1,z_2,\ldots,z_{i}\}$, we have $\phi_{i}(z_1,z_2,\ldots,z_{i},y_k)\neq 0$. Since $\phi$ is a homomorphism and the set $U$ is  a minimal generating set for $L$, thus the set   $\{\phi_{i}(z_1,z_2,\ldots,z_{i},y_k)~|~ k=1,\ldots, n+m-r-s-i\}$ is a linearly independent set. From Corollary \ref{cor3.2}, we have Im$(\phi_{i})\subseteq {\rm{ker}} (\lambda_{i}) $, and  ${\rm{ker}} (\lambda_{i})\neq 0$ for all $2\leq i\leq c$. Thus, we have  
	\begin{equation}\label{eq3.1}
		\dim (ker(\lambda_{i}))\geq n+m-r-s-i.
	\end{equation}

Now using (\ref{eq3.1}) in (\ref{eq2.4}), we get 
\begin{align*}
\dim(\mathcal{M}(L)) &= \dim(\mathcal{M}(L/\gamma_2(L)))  + \dim(\gamma_2(L))(\dim (\frac{L}{\gamma_2(L)})-1) - \sum_{i=2}^{c}\dim (ker (\lambda_{i}))\\
&\leq\dim(\mathcal{M}(L/\gamma_2(L)))  + \dim(\gamma_2(L))(\dim (\frac{L}{\gamma_2(L)})-1) - \sum_{i=2}^{l}(m+n-r-s-i).
\end{align*}
By using Lemma \ref{th3.3}, 
\begin{align*}
\dim(\mathcal{M}(L/\gamma_2(L)))&=\frac{1}{2}[(m+n-r-s)^2+n-s-m+r]\\
&=\frac{1}{2}[(m+n-r-s)^2+2(n-s)-(m+n-r-s)]\\
&=\frac{1}{2}[(m+n-r-s)(m+n-r-s-1)+2(n-s)].
\end{align*}
Therefore, 
\begin{align*}
	\dim(\mathcal{M}(L))\leq
	&~\frac{1}{2}[(m+n-r-s)(m+n-r-s-1)+2(n-s)]+(r+s)(m+n-r-s-1) \\ &-\sum_{i=2}^{l}(m+n-r-s-i)\\
	&\leq\frac{1}{2}[(m+n-r-s-1)(m+n+r+s)+2(n-s)]- \sum_{i=2}^{l}(m+n-r-s-i)\\
	&\leq \frac{1}{2}[(m + n - r - s)(m + n + r + s) + (n - m - r - 3s)]-\sum_{i=2}^{l}(m+n-r-s-i).
\end{align*}
\end{proof}

\textbf{Disclosure statement:} All authors declare that they have no conflicts of interest that are relevant to the content of this article.

\end{document}